\documentclass[12pt]{amsart}

\usepackage{geometry}
\usepackage{amssymb,hyperref}
\usepackage{graphicx,tikz} 

\newtheorem{theorem}{Theorem}[section]

\newtheorem{proposition}[theorem]{Proposition}
\newtheorem{corollary}[theorem]{Corollary}

\theoremstyle{definition}
\newtheorem{definition}[theorem]{Definition}

\newtheorem{example}[theorem]{Example}
\newtheorem{question}[theorem]{Question}
\newtheorem{remark}[theorem]{Remark}

\newcommand{\R}{\mathbb{R}}
\newcommand{\Z}{\mathbb{Z}}
\newcommand{\nvol}{\mathrm{Vol}}
\newcommand{\conv}{\mathrm{conv}}

\newcommand{\des}{\mathrm{des}}
\newcommand{\eul}{\mathrm{Eul}}
\newcommand{\hdis}{\mathrm{hdis}}
\newcommand{\Var}{\operatorname{Var}}
\newcommand{\E}{\mathbb{E}}
\newcommand{\Prob}{\mathbb{P}}

\usepackage[colorinlistoftodos]{todonotes}

\definecolor{munsell}{rgb}{0.0, 0.5, 0.69}

\title{Ehrhart $h^*$-distributions}

\author[Braun]{Benjamin Braun}
\address{Department of Mathematics\\
         University of Kentucky
}
\email{benjamin.braun@uky.edu}
\urladdr{https://sites.google.com/view/braunmath/}

\author[Hlavacek]{Max Hlavacek}
\address{Department of Mathematics and Statistics\\
Pomona College
}
\email{mhap2023@pomona.edu}
\urladdr{https://maxhlav.github.io/}

\author[Meza]{Cesar J. Meza}
\address{Department of Mathematics\\
Washington University in St. Louis
}
\email{c.j.meza@wustl.edu}
\urladdr{https://cesarjmeza.github.io/}

\author[Morales]{Santiago Morales}
\address{Department of Mathematics \\
University of California, Davis
}
\email{moralesduarte@ucdavis.edu}
\urladdr{https://smoralesduarte.github.io}

\author[Vindas-Mel\'{e}ndez]{Andr\'{e}s R. Vindas-Mel\'{e}ndez}
\address{Department of Mathematics\\
Harvey Mudd College
}
\email{avindasmelendez@g.hmc.edu}
\urladdr{https://math.hmc.edu/arvm}

\date{16 July 2026}

\begin{document}

\begin{abstract}
    Every polynomial with real non-negative coefficients yields a finite probability distribution after normalization.
    The Ehrhart $h^*$-polynomial of a lattice polytope $P$ is a non-negative integer polynomial that encodes the integer-point counts for positive integer dilations of $P$.
    We study the corresponding finite distributions, which we call $h^*$-distributions.
    We determine the mean and variance of these distributions, establish a connection between higher moments and Ehrhart polynomial coefficients, and study their cluster points in the $d$-dimensional probability simplex.
    We consider the special case of real-rooted $h^*$-distributions, applying existing tail bounds to obtain new linear inequalities for the coefficients of real-rooted $h^*$-polynomials arising from reflexive polytopes.
    We conclude by establishing sufficient conditions under which a sequence of real-rooted $h^*$-distributions is asymptotically normal, and we apply our results to various families of polytopes, including zonotopes and Pitman-Stanley polytopes.
\end{abstract}

\thanks{BB was partially supported by NSF award DMS-2450299.
ARVM was partially supported by NSF award DMS-2532321.
The authors thank Daniel Hwang, Martina Juhnke, Selvi Kara, Brittney Marsters, Akiyoshi Tsuchiya, and Zoe Wellner for thoughtful conversations at the beginning of this project.
The authors thank the Institute for Pure and Applied Mathematics for hosting the workshop ``Computational Interactions between Algebra, Combinatorics, and Discrete Geometry'', where this project was initiated.}

\maketitle

\section{Introduction}\label{sec:introduction}

For a $d$-dimensional lattice polytope $P$, i.e., a convex polytope with vertices in $\Z^n$ and affine span of dimension $d$, Ehrhart theory is the study of the sequence $|tP\cap \Z^n|$ for $t\geq 0$.
A fundamental result of Ehrhart theory is that this entire sequence is determined by a polynomial of degree at most $d$ called the $h^*$-polynomial of $P$, and that the (normalized) volume of $P$ is given by the sum of the $h^*$-polynomial coefficients.
The study of Ehrhart $h^*$-polynomials is a major focus of geometric and algebraic combinatorics, with connections to commutative algebra, algebraic geometry, and number theory~\cite{BeckRobins,BrunsGubeladze,FHsurvey,Hibibook}.

There is a long history in Ehrhart theory of the study of $h^*$-polynomial properties motivated by probability theory, e.g., unimodality, log-concavity, and real-rootedness~\cite{brandensurvey,braunsurvey,FHsurvey}.
For any polynomial $f(t)$ with real non-negative coefficients, dividing the coefficients by $f(1)$ yields a finite probability distribution.
While there has been significant effort put into the study of $h^*$-polynomials, a systematic study of probability distributions arising from normalized $h^*$-polynomials has not to the knowledge of the authors been done.
However, there are several reasons to consider this distributional view.

First, structure that is masked by differences in volume between polytopes becomes more clear.
For example, consider the line plots of $h^*$-vectors, e.g., coefficient vectors for $h^*$-polynomials, for all flow polytopes of full acyclic directed graphs with $9$ vertices~\cite{ampleframingflow}, which are plotted in Figure~\ref{fig:9vertexfulldag}.
Because of the wide variation in the normalized volume of these polytopes, it is not clear how to compare these $h^*$-vectors, and the plot is not particularly informative.
However, when the corresponding distributions for these vectors are plotted in Figure~\ref{fig:9vertexfulldagdist}, it is immediately clear that these have similar structure.

\begin{figure}
  \centering
\begin{minipage}{0.49\textwidth}
\centering
  \includegraphics[width=\textwidth]{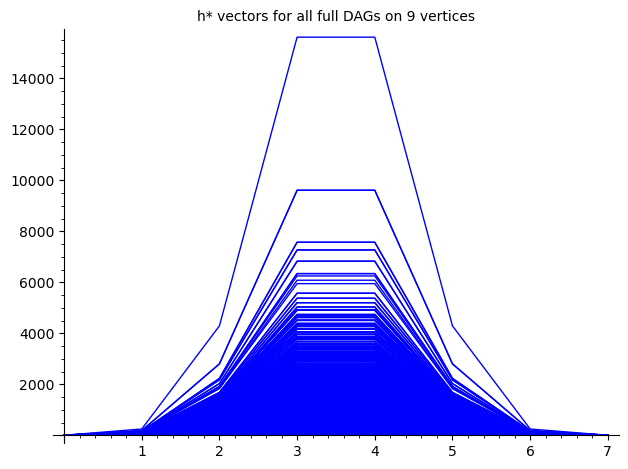}
  \caption{Line plots of all \textbf{$h^*$-vectors} for flow polytopes of full DAGs with 9 vertices.
  The horizontal axis is the index of $h^*$ while the vertical axis is the value of $h^*_j$.}
  \label{fig:9vertexfulldag}
  \end{minipage}
  \hfill\begin{minipage}{0.49\textwidth}
    \centering
    \includegraphics[width=\textwidth]{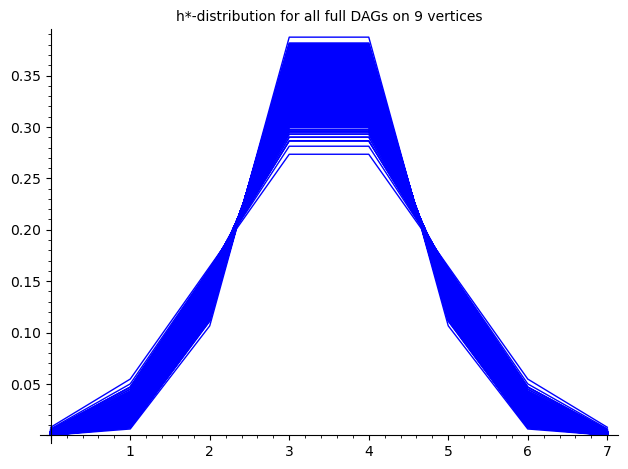}
    
    \caption{Line plots of all \textbf{$h^*$-distributions} for flow polytopes of full DAGs with 9 vertices.
    The horizontal axis is the index of $h^*/\sum_{i=0}^dh^*_i$ while the vertical axis is the value of $h^*_j/\sum_{i=0}^dh^*_i$.}
    
    \label{fig:9vertexfulldagdist}
\end{minipage}
\end{figure}

Second, taking an explicit distributional view of $h^*$-polynomials allows for the application of tools from probability theory to their study.
In his 1997 survey article regarding real-rooted distributions arising from enumeration problems, Pitman~\cite{pitman} emphasized that there is a long history in statistics of studying the distribution of the number of successes in independent trials, which is equivalent to studying real-rooted distributions, and that various results in the statistics literature were unfamiliar to combinatorialists.
In our experience, these results remain unfamiliar to many researchers in Ehrhart theory, and we believe it would be beneficial to bring increased awareness to these and similar connections with probability and statistics.

With this as motivation, our contributions in this paper are the following.

\begin{itemize}
    \item We introduce $h^*$-distributions in Subsection~\ref{ssec:distributions} and determine formulas for their mean and variance in Subsection~\ref{ssec:meanvariance}.
    \item The normalization process allows $h^*$-distributions to be directly compared, leading to the study of cluster points.
    We study cluster points in Subsection~\ref{ssec:cluster} and apply our results to a combinatorial problem involving permutation statistics in Subsection~\ref{ssec:combinatorialapplication}.
    \item We apply tail bounds for real-rooted distributions to obtain inequalities for $h^*$-vectors in Subsection~\ref{ssec:tailboundapplications}, and we establish sufficient conditions for a sequence of real-rooted $h^*$-distributions to be asymptotically normal in Subsection~\ref{ssec:asymptoticnormality}.
\end{itemize}

We conclude this work with several further directions for research in Section~\ref{sec:furtherdirections}.

\section{Background and notation}\label{sec:background}

\subsection{Ehrhart theory}\label{subsec:background}
Let $P$ denote a $d$-dimensional lattice polytope in $\R^n$, i.e., the convex hull of a finite set of vectors in $\Z^n$ having affine span of dimension $d$.
Letting $tP$ denote the dilation of $P$ by $t\in \R$, the \emph{Ehrhart series} of $P$ is
\[
1+\sum_{t\in \Z_{\geq 1}}|tP\cap \Z^n|z^t = \frac{\sum_{j=0}^dh_j^*z^j}{(1-z)^{d+1}},
\]
where rationality of the power series is due to Ehrhart~\cite{Ehrhart}.
A result due to Stanley~\cite{StanleyDecompositions} is that each $h_j^*\in \Z_{\geq 0}$, which is referred to as \emph{$h^*$-non-negativity}.
We call the polynomial $h^*(P;z):=\sum_{j=0}^dh_j^*z^j$ the \emph{$h^*$-polynomial} of $P$ and the vector $(h_j^*:j=0,\ldots,d)$ the \emph{$h^*$-vector} of $P$.
A basic result of Ehrhart theory is that 
\[
\sum_{j=0}^dh_j^*=\nvol(P):=d!\mathrm{vol}(P),
\]
where $\mathrm{vol}(P)$ is the Euclidean volume of $P$ with respect to an integer basis of the integer lattice contained in the affine span of $P$.
We call $\nvol(P)$ the \emph{normalized volume} of $P$.
Because the Ehrhart series is a rational function with denominator $(1-z)^{d+1}$ and numerator polynomial of degree at most $d$, the function
\[
L_P(t):=|tP\cap \Z^n|
\]
is given by a polynomial of degree $d$, which
we call the \emph{Ehrhart polynomial} of $P$ and represent as
\[
  L_P(t)=|tP\cap\Z^n|=c_dt^d+c_{d-1}t^{d-1}+c_{d-2}t^{d-2}+\cdots+c_0 \, .
\]
A standard fact in Ehrhart theory is that
\[
\nvol(P)=d!c_d\, .
\]

The coefficient $c_{d-1}$ admits a geometric interpretation involving the normalized volumes of facets of $P$.
Given a lattice polytope $P$, the \emph{normalized boundary volume} of $P$ is defined to be
\[
B(P):=\sum_{F\text{ facets of }P}\nvol(F)\, .
\]
It is a well-known~\cite[Theorem~5.6]{BeckRobins} result in Ehrhart theory that for a $d$-dimensional lattice polytope $P$, we have 
\[
(d-1)!c_{d-1}=\frac{1}{2}B(P)\, .
\]

A $d$-dimensional lattice polytope $P$ is called \emph{reflexive} if $h^*_i=h^*_{d-i}$ for all $i=0,\ldots,d$.
Thus, reflexive polytopes are those with full-degree $h^*$-polynomials having palindromic coefficients.
Reflexive polytopes are known to have every facet at lattice distance $1$ from a unique interior point, and as a result, reflexive polytopes satisfy
\[
\nvol(P)=B(P) \, .
\]

\subsection{Finite probability distributions}

All background in this subsection is found in the survey by Pitman~\cite{pitman}; further details can be found in Pitman's article and the references therein.
Let $[0,d]:=\{0,1,\ldots,d\}$ and let $\Prob:[0,d]\to \R_{\geq 0}$ be a finite probability distribution.
For a finite set $\Omega$, we denote by $X:\Omega\to [0,d]$ a random variable and we write $X\sim\Prob$ if $X$ has distribution $\Prob$.
We denote by $\E[X]=\sum_{i=0}^di\, \Prob(i)$ and $\Var[X]=\E[X^2]-\E[X]^2$ the expected value and variance of $X$, respectively.

The distribution of the number of successes in $d$ independent Bernoulli trials having probabilities $0\leq p_i\leq 1$ for $i=1,\ldots,d$, respectively, is given by the coefficient vector of the polynomial
\[
\prod_{i=1}^d((1-p_i)+p_iz)^d \, .
\]
Thus, the roots of this polynomial are given by $-(1-p_i)/p_i$ for $i=1,\ldots,d$.

Many important finite probability distributions arise in combinatorics from a random variable $X$ defined by a statistic on a finite set of objects $\Omega$.
In many cases, setting $a_k:=|\{j\in \Omega:X(j)=k\}|$ and $A(z):=\sum_{k\geq 0}a_kz^k$ yields a real-rooted polynomial, in which case we say the distribution given by the normalized sequence
\[
\frac{1}{A(1)}(a_0,\ldots,a_d)
\]
is a \emph{real-rooted} distribution.
It is straightforward to verify that a finite distribution is real-rooted if and only if it arises from a random variable given by a sum of independent Bernoulli trials.
It is known that for a random variable $X$ with a real-rooted distribution arising from a sum of independent trials with probabilities $p_1,\ldots,p_d$, we have

\[
\E[X]=\sum_ip_i
\]
and
\[
\Var[X]=\sum_{i=0}^dp_i(1-p_i) \, .
\]

Define $\Prob[a,b]:=\sum_{i=a}^b\Prob(i)$.
Let $P_{d,p}$ be the binomial distribution, i.e., the distribution arising from $d$ Bernoulli trials where every $p_i$ is equal to $p$.
Given a real-rooted distribution $\Prob$ on $[0,d]$ with $X\sim \Prob$, set $p=\E[X]/d$.
Hoeffding's inequalities state that for all integers $n$ and $m$ with $0\leq m\leq \E[X]-1$ and $\E[X]+1\leq n\leq d$, we have
\begin{equation}\label{eq:hoeffding}
\Prob[0,m]\leq P_{d,p}[0,m]
\text{ and }
\Prob[n,d]\leq P_{d,p}[n,d] \, .
\end{equation}
As Pitman~\cite[Section~2]{pitman} observes, this implies that among all real-rooted distributions on $[0,d]$ with a fixed mean, the binomial distribution is the ``most spread out''.
Combining these inequalities with binomial tail probability bounds yields the following large deviation bounds for $\E[X]+1\leq n\leq d$:
\begin{equation}
    \label{eq:largedeviation}
    \Prob[n,d]\leq \left(\frac{\E[X]}{n}\right)^n\left(\frac{d-\E[X]}{d-n}\right)^{d-n}.
\end{equation}
The same function is an upper bound on $\Prob[0,m]$ for $0\leq m\leq \E[X]$ with $n$ replaced by $m$.

Let 
\[
\phi(x) = \frac{1}{\sqrt{2\pi}}e^{-x^2/2}
\]
be the normal density function, and define
\[
\Phi(z)=\int_{-\infty}^z\phi(x)dx \, .
\]
For any sequence of distributions $Q_d\sim X_d$ on $\{0,\ldots,d\}$ for $d=1,2,\ldots$ with means $\E[X_d]$ and variances $\Var[X_d]$, we say the sequence of distributions $Q_d$ is \emph{asymptotically normal} if 
\[
\max_{0\leq k\leq d}\left| Q_d[0,k] - \Phi\left(\frac{k-\E[X_d]}{\sqrt{\Var[X_d]}}\right) \right| \to 0 \text{ as }d \to \infty \, .
\]
If $Q_d\sim X_d$ is a sequence of real-rooted distributions, then the sequence is asymptotically normal if and only if 
\begin{equation}
    \label{eq:rrasymp}
\sqrt{\Var[X_d]}\to \infty \text{ as }d\to \infty \, .
\end{equation}

\section{$h^*$-distributions and their properties}\label{sec:distributions}

\subsection{$h^*$-distributions}\label{ssec:distributions}
We now define the main object of study in this work.

\begin{definition}
    Given a lattice polytope $P$, the \emph{$h^*$-distribution} of $P$ is the finite probability distribution given by 
    \[
    \hdis(P):=
    \frac{1}{\nvol(P)} (h_0^*,h_1^*,\ldots,h_d^*) \, ,
    \]
    where the probability of outcome $j$ is given by $h_j^*/\nvol(P)$. 
    More generally, given a polynomial $f(x)$ with non-negative real coefficients, the \emph{distribution} associated to $f$ is the discrete probability distribution given by the coefficient vector of
    \(\frac{f(x)}{f(1)}\).
\end{definition}

\begin{example}\label{ex:1d}
    The Ehrhart series of a lattice line segment in $\R^{1}$ of length $\ell \geq 1$ is 
    \begin{equation*}
        \sum_{t \geq 0} (\ell{t} + 1)z^{t} = \frac{1 + (\ell-1)z}{(1-z)^{d+1}}.
    \end{equation*}
    It follows that the $h^{*}$-distribution of a lattice line segment in $\R^{1}$ is $(\frac{1}{\ell}, \frac{\ell-1}{\ell})$, where $\ell \geq 1$.
\end{example}

\begin{example}\label{ex:2d}
    In 1976, Scott \cite{Scott} classified the $h^*$-vectors $(h^*_0, h^*_1, h^*_2)$ of two-dimensional polytopes, showing that they must satisfy one of the following conditions:
    \begin{enumerate}
        \item $(1, h^*_1, 0)$ with $h^*_1 \geq 0$
        \item $1 \leq h^*_2 \leq h^*_1 \leq 3h^*_2 + 3$
        \item $(h_0^*,h_1^*,h_2^*)=(1,7,1)$
    \end{enumerate}
    Note that only the triangle $\conv(0, 3e_{1}, 3e_{2})$, where $e_{1}$, $e_{2}$ are basis vectors of the ambient lattice and $0$ is the origin, has $h^*$-vector $(1,7,1)$~\cite{Scott, Balletti-Higashitani}.
    For an $h^*$-vector with normalized volume $V$, Figure~\ref{fig:2d} is a plot of $(h_1^*/V, h_2^*/V)$ for all $h^*$ distributions of two-dimensional lattice polygons with $h^*_2 < 200$.
    Because $h_0^*=1/V$ is small for large $V$, it makes sense for visualization purposes to project away the $0$-th coordinate.
    The Eulerian point $(1/2,1/2)$ is a red square, and has the highest $h_2^*/V$ coefficient.
    The point $(1/3,1/3)$, corresponding to the barycenter of the probability simplex, is the orange ``x'' on the left.
    The distribution for the vector $(1,7,1)$ is a black ``Y''  isolated in the lower right.
    Note that on the x-axis are a sequence of points clustering toward $(1,0)$, which are the $h^*$ distributions of $1$-dimensional lattice segments, obtained in dimension two through the lattice pyramid operation. 
\end{example}

\begin{figure}
    \centering
    \includegraphics[width=0.8\linewidth]{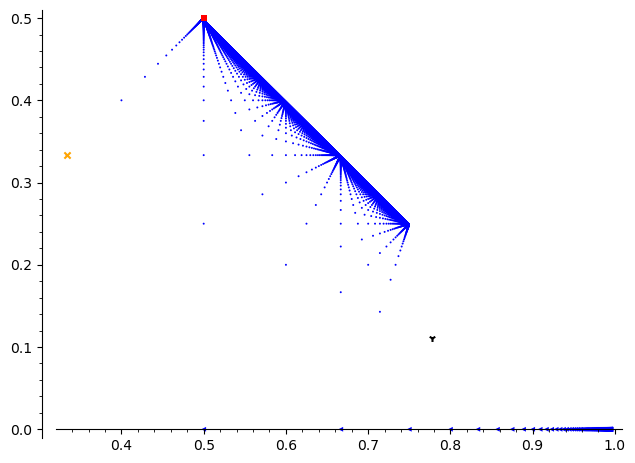}
    \caption{A plot of $(h_1^*/V, h_2^*/V)$ for all $h^*$ distributions of two-dimensional lattice polygons with $h^*_2 < 200$, as described in Example~\ref{ex:2d}.
    The horizontal axis corresponds to the value $h^*_1/V$ while the vertical axis corresponds to the value $h^*_2/V$.}
    \label{fig:2d}
\end{figure}

In general, $h^*$-distributions do not arise from combinatorial random variables in a canonical way.
However, for lattice simplices, the $h^*$-distribution does arise from a statistic on a set of objects.
Given a lattice simplex $P$ with vertices $\mathbf{v}_0,\ldots,\mathbf{v}_d$, the \emph{half-open fundamental parallelepiped} for $P$ is
\[
\Pi_P:=\left\{\sum_{i=0}^d\lambda_i(1,\mathbf{v}_i):0\leq \lambda_i<1\right\} \, .
\]

\begin{proposition}
    \label{prop:simplicesdistribution}
    Given a lattice simplex $P$, let $X_P:\Pi_P\cap \Z^{d+1}\to \Z$ be the random variable defined by 
    \[
    X_P((y_0,y_1,\ldots,y_d))=y_0 \, ,
    \]
    i.e., $X_P$ is the height of $(y_0,y_1,\ldots,y_d)$ in the cone over $P$.
    Then $X_P\sim \hdis(P)$.
\end{proposition}

\begin{proof}
    It is known~\cite[Chapter~3]{BeckRobins} that for a lattice simplex $h_j^*$ is equal to the number of lattice points in the fundamental parallelepiped with first coordinate $j$, i.e., height $j$.
    The result follows.
\end{proof}

Thus, for a lattice simplex, the $h^*$-distribution encodes the probability of selecting a random fundamental parallepiped point at a given height.

\subsection{Mean and Variance}\label{ssec:meanvariance}

We can express the mean and variance for an $h^*$-distribution using the coefficients of $L_P(t)$ and the normalized boundary volume.

\begin{theorem}\label{thm:meanvariance}
Let \(P\) be a \(d\)-dimensional lattice polytope and consider a random variable \(X_P\sim\hdis(P)\).
Then
\begin{equation}\label{eq:mean}
  \E[X_P]
  =
  \frac{d+1}{2}
  -
  \frac{(d-1)!c_{d-1}}{\nvol(P)}=\frac{d+1}{2}-\frac{B(P)}{2\nvol(P)} \, .
\end{equation}
If \(d\ge2\), then
\[
  \E[X_P^2]
  =
  \frac{2(d-2)!c_{d-2}}{\nvol(P)}
  +(d+1)\E[X_P]
  -\frac{(d+1)(3d+2)}{12},
\]
and therefore
\begin{equation}\label{eq:variance}
  \Var[X_P]
  =
  \frac{2(d-2)!c_{d-2}}{\nvol(P)}
  +(d+1)\E[X_P]
  -\frac{(d+1)(3d+2)}{12}
  -\E[X_P]^2.
\end{equation}
\end{theorem}

\begin{proof}
The Ehrhart polynomial has the binomial-basis expansion
\[
  L_P(t)=\sum_{j=0}^d h_j^*\binom{t+d-j}{d}.
\]
The coefficient of \(t^{d-1}\) in \(\binom{t+d-j}{d}\) is
\[
  \frac{(d+1)/2-j}{(d-1)!}.
\]
Equating coefficients of \(t^{d-1}\) on each side of the first displayed equation gives
\[
  c_{d-1}=\frac{\nvol(P)}{(d-1)!}\left(\frac{d+1}{2}-\E[X_P]\right),
\]
which is the mean formula~\eqref{eq:mean}.

For the variance, the coefficient of \(t^{d-2}\) in \(\binom{t+d-j}{d}\) is
\[
  \frac{12j^2-12(d+1)j+(d+1)(3d+2)}{24(d-2)!}.
\]
Comparing the coefficients of \(t^{d-2}\) yields
\[
  c_{d-2}=\frac{\nvol(P)}{24(d-2)!}
  \left(12\E[X_P^2]-12(d+1)\E[X_P]+(d+1)(3d+2)\right),
\]
from which the asserted formula~\eqref{eq:variance} follows. 
The boundary-volume identity follows from the well-known equality \(2(d-1)!c_{d-1}=B(P)\)~\cite{BeckRobins}.
\end{proof}

\begin{example}
    \label{ex:unimodularmean}
    Given a unimodular simplex $P$ of dimension $d$ and $X_P\sim \hdis(P)$, we have $\nvol(P)=1$ and $B(P)=d+1$.
    Thus, 
    \[
    \E[X_P]=\frac{d+1}{2}-\frac{d+1}{2\cdot 1} = 0 \, .
    \]
    As $h^*(P;z)=1$, this is correct.
\end{example}

\begin{example}
    \label{ex:reflexive}
    Given a $d$-dimensional reflexive polytope $P$ and $X_P\sim \hdis(P)$, we have that $B(P)/2\nvol(P)=1/2$.
    Thus, 
    \[
    \E[X_P]=\frac{d+1}{2}-\frac{1}{2}=\frac{d}{2}\, .
    \]
    This agrees with the fact that the $h^*$-vector of $P$ is symmetric with center of symmetry $\frac{d}{2}$.
\end{example}

Example~\ref{ex:reflexive} implies that the variance of $h^*$-distributions for reflexive polytopes depends only on the dimension and the coefficients $c_d$ and $c_{d-2}$ of $L_P(t)$, recorded precisely in the following corollary.

\begin{corollary}
    \label{cor:varreflexives}
    Let $P$ be a reflexive polytope of dimension $d$, and let $X_P\sim \hdis(P)$.
    Then 
    \[
    \E[X_P] = d/2
    \]
    and 
    \[
      \Var[X_P]
      =
      \frac{2}{d(d-1)}\frac{c_{d-2}}{c_d}
      +\frac{d-2}{12}.
    \]
\end{corollary}

\begin{proof}
    Apply $\nvol(P)=d!c_d$ and $\E[X_P]=d/2$ to Theorem~\ref{thm:meanvariance}, then simplify the resulting expression.
\end{proof}

Using Theorem~\ref{thm:meanvariance}, we can also obtain expressions for the mean and variance of integer dilates of $P$.

\begin{corollary}\label{cor:momentsdilates}
Let \(d\ge2\), and define
\[
  \alpha(P)=\frac{(d-1)!c_{d-1}}{\nvol(P)},
  \qquad
  \beta(P)=\frac{2(d-2)!c_{d-2}}{\nvol(P)}.
\]
Then for positive integer dilates $tP$ of a $d$-dimensional lattice polytope $P$, we have
\[
  \E[X_{tP}]=\frac{d+1}{2}-\frac{\alpha(P)}{t}
\]
and
\[
  \Var[X_{tP}]=\frac{d+1}{12}+\frac{\beta(P)-\alpha(P)^2}{t^2} \, .
\]
\end{corollary}

\begin{proof}
The Ehrhart coefficients of \(tP\) are \(t^ic_i\), while \(\nvol(tP)=t^d\nvol(P)\).  
The result follows from substituting these into Theorem~\ref{thm:meanvariance}.
\end{proof}

Recall that a \emph{lattice zonotope} is a Minkowski sum $Z=\sum_{i=1}^m [0,\mathbf{v}_i]$ of lattice segments, where $\mathbf{v}_1,\dots, \mathbf{v}_m \in \Z^d$ and $[0,\mathbf{v}_i]$ denotes the line segment from $0$ to $\mathbf{v}_i$.
The following formula due to Stanley makes the mean and variance of the $h^*$-distribution of a zonotope explicitly computable from the generators of the lattice segments.

\begin{theorem}[Stanley~\cite{stanleyzonotopes}, Theorem 2.2]\label{thm:stanleyzonotopes}
The Ehrhart polynomial of a lattice zonotope $Z$ is given by 
\begin{equation*}
L_Z(t)=\sum_{S}m(S)t^{|S|} \, ,
\end{equation*}
where $S$ ranges over all linearly independent subsets of $\{\mathbf{v}_1,\dots,\mathbf{v}_m\}$, and $m(S)$ is the greatest common divisor of all minors of size $|S|$ of the matrix whose columns are the elements of $S$. 
\end{theorem}

\begin{proposition}\label{prop:zonotope-mean-variance}
   For  $Z=\sum_{i=1}^m [0,\mathbf{v}_i]$ a lattice $d$-zonotope and $X_Z\sim \hdis(Z)$, we have 
\[ 
\E[X_Z] = \frac{d+1}{2} - \frac{1}{d}\left(\frac{\sum_{|S|=d-1}m(S)}{\sum_{|S|=d}m(S)}\right)\, .
\]
For $d\geq 2$, we have 
\[
\Var[X_Z] = \frac{2}{d(d-1)}\left(\frac{\sum_{|S|=d-2}m(S)}{\sum_{|S|=d}m(S)}\right)+\frac{d+1}{12}-\left( \frac{1}{d}\left(\frac{\sum_{|S|=d-1}m(S)}{\sum_{|S|=d}m(S)}\right)\right)^2\, .
\]
\end{proposition}

\begin{proof}
    By Theorem \ref{thm:stanleyzonotopes}, $c_k=\sum_{|S|=k}m(S)$ is the coefficient of $t^k$ in $L_Z(t)$, and $\nvol(Z)=d!\,c_d$. 
    Theorem \ref{thm:meanvariance} gives 
    \[\E[X_Z] = \frac{(d+1)}{2}- \frac{(d-1)!\; c_{d-1}}{\nvol(Z)}= \frac{d+1}{2}-\frac{1}{d}\left(\frac{c_{d-1}}{c_d}\right).\]
    For the variance, Theorem \ref{thm:meanvariance} gives 
    \[\Var[X_Z]=\frac{2(d-2)!c_{d-2}}{\nvol(Z)}+(d+1)\E[X_Z]-\frac{(d+1)(3d+2)}{12}-\E[X_Z]^2.\]
    Substituting $\E[X_Z]$ and simplifying, we obtain the desired variance formula.
\end{proof}

\begin{example}\label{ex:unimodularzonotope}
    When $Z$ is unimodular, every nonzero maximal minor equals $\pm 1$, and the formula $\sum_{|S|=k}m(S)$ counts the number of independent sets of size $k$ in the linear matroid represented by the generating vectors of $Z$ \cite{stanleyzonotopes}.
    Writing $M$ for that matroid and $I_k(M)$ for the number of independent sets of size $k$, we have that $\nvol(Z)=d! I_d(M)$ and so we can rewrite $\E[X_Z]$ as 
    \[
    \E[X_Z]= \frac{d+1}{2}- \frac{1}{d}\left(\frac{I_{d-1}(M)}{I_d(M)}\right)\, .
    \]
    \end{example}
    
    \begin{example}\label{ex:graphicalzonotope}For the zonotope $Z_G$ generated by the vectors $\mathbf{e}_i-\mathbf{e}_j$ over the edges $\{i<j\}$ of a connected graph $G$, the matroid $M$ is the graphic matroid of $G$, so $I_d(M)$ is the number of spanning trees of $G$ and $I_{d-1}(M)$ is the number of spanning forests of $G$ with exactly two components. 
    For example, when $G=K_4$, we have that 
    \[
    \E[X_{Z_{K_4}}] = 2-\frac{1}{3}\cdot \frac{15}{16} = \frac{27}{16}\, .
    \]
\end{example}

We end this subsection with the observation that the moments of $X_P$ are determined by the coefficients of the Ehrhart polynomial of $P$ and vice versa.

\begin{theorem}
    \label{thm:momentscoefficients}
    Consider a lattice polytope $P$ with $X_P\sim \hdis(P)$ and Ehrhart polynomial $L_P(t)=\sum_{j=0}^dc_jt^j$.
    For any $0\leq k\leq d$, the normalized volume and the values $\E[X_P^r]$ for $r=0,1,\ldots,k$ are determined by $c_{d-r}$ for $r=0,1,\ldots,k$, and vice versa.
\end{theorem}

\begin{proof}
    Let $e_{d,k}(x_1,\ldots,x_d)$ denote the degree $k$ elementary symmetric function on $d$ variables.
    Observe that 
    \begin{align*}
    L_P(t) & =\sum_{j=0}^dh^*_j\binom{t+d-j}{d} \\
    & = \sum_{j=0}^d\frac{1}{d!}h^*_j\sum_{k=0}^de_{d,k}(1-j,2-j,\ldots,d-j)t^{d-k}\\
    &= \sum_{j=0}^d\frac{1}{d!}h^*_j\sum_{k=0}^d \left( \sum_{\ell=0}^k \binom{d-k+\ell}{\ell} e_{d,k-\ell}(1,2,\ldots,d)(-1)^{\ell}j^{\ell}\right) t^{d-k}\\
    & = \sum_{k=0}^d\left[ \sum_{\ell=0}^k\nvol(P)\binom{d-k+\ell}{\ell}e_{d,k-\ell}(1,2,\ldots,d)\frac{1}{d!}(-1)^{\ell} \left( \sum_{j=0}^d \frac{h_j^*}{\nvol(P)}j^{\ell}\right) \right]t^{d-k}\\
    &= \sum_{k=0}^d\left[ \sum_{\ell=0}^k\nvol(P)\binom{d-k+\ell}{\ell}e_{d,k-\ell}(1,2,\ldots,d)\frac{1}{d!}(-1)^{\ell} \E[X_P^\ell] \right]t^{d-k} \, .
    \end{align*}
    Thus, the coefficient $c_{d-k}$ is determined by $\nvol(P)$ and linear combinations of $\E[X_P^\ell]$ for $\ell=0,\ldots,k$, and the linear system relating these values is triangular and invertible.
    Hence, our proof is complete.
\end{proof}

Note that when $k=0$ in the proof above, we have that $c_d=\nvol(P)/d!$, as expected.

\begin{remark}\label{rem:endquasi}
    When $P$ is a rational polytope, the function $L_P(t)$ is a quasi-polynomial rather than a polynomial.
    The higher moments of distributions defined by $h$-polynomials of arbitrary quasi-polynomials play a key role in the classification of eventually nondecreasing quasi-polynomials~\cite{endquasipoly} when the $h$-vector is non-negative.
\end{remark}

\subsection{Cluster points}\label{ssec:cluster}

Observe that every $h^*$-distribution of a $d$-dimensional polytope is a point in the $d$-dimensional probability simplex $\conv(e_0,e_1,\ldots,e_d)$.
The following proposition shows that every $h^*$-distribution is an isolated point in this simplex, and cluster points must lie on a specific facet of the probability simplex.

\begin{proposition}\label{prop:isolatedpoints}
    If $(h_0,\ldots,h_d)$ is an $h^*$-distribution for a $d$-dimensional lattice polytope $P$, then the distribution is isolated in the Euclidean topology in the probability simplex.
    Any cluster point of the set of $h^*$-distributions must satisfy $h_0=0$.
\end{proposition}

\begin{proof}
    Note that $h_0=1/\nvol(P)$ since $h_0^*=1$ for every lattice polytope.
    Lagarias and Ziegler~\cite{LagariasZiegler} proved that for fixed positive integers $d$ and $V$, there are finitely many unimodular-equivalence classes of $d$-dimensional lattice simplices with normalized volume at most $V$. 
    Hence, for any fixed $V$ and $d$, there are only finitely many $h^*$-distributions in the $d$-dimensional probability simplex with $h_0>1/V$.
    Therefore, every $h^*$-distribution is an isolated point.

    Further, the only way for a cluster point to be obtained is as a limit of an infinite sequence of $h^*$-distributions, and thus each cluster point can be obtained as an infinite sequence of lattice polytopes with strictly increasing volumes. 
    Thus, the sequence of $0$-th coefficients goes to zero.
\end{proof}

Given that every $h^*$-distribution is an isolated point from a metric perspective, this motivates the question of when an infinite sequence or set of $h^*$-distributions has a cluster point, as illustrated by the following three examples.

\begin{example}
For every lattice simplex $S$, there exists a unique matrix $H$ representing the vertices of $S$ up to unimodular equivalence \cite{Schrivjer}, which are called the \emph{Hermite normal form of} $S$.
One special family of simplices are one-row Hermite normal form simplices, which are those unimodularly equivalent to the convex hull of the rows of an integer matrix as in~\eqref{eq:onerowintro}, specified by parameters $a_1, \dots, a_{d-1}, N$ with \(0\leq a_i< N\) for all \(i\). 
Note that the normalized volume of these simplices is exactly the parameter $N$.

\begin{equation}\label{eq:onerowintro}
H = 
\begin{bmatrix}
		 0 &  0 &  0 & \cdots &  0 &  0 &  0\\
		 1 &  0 &  0 & \cdots &  0 &  0 &  0\\
		 0 &  1 &  0 & \cdots &  0 &  0 &  0\\
		0 &  0 &  1 & \cdots &  0 &  0 &  0\\
		\vdots & \vdots & \vdots & \ddots & \vdots & \vdots & \vdots \\
		 0 &  0 &  0 & \cdots &  1 &  0 &  0\\
		 0 &  0 &  0 & \cdots &  0 &  1 &  0\\
		 a_1 & a_2 & a_3 & \cdots & a_{d-2} & a_{d-1} & N
	\end{bmatrix}.
\end{equation}

For a fixed vector $(a_1,
\ldots,a_{d-1})$, we define $S_N$ to be the simplex defined as the convex hull of the rows of~\eqref{eq:onerowintro}.
For 
\[
M= \operatorname{lcm}\left(a_1,a_2,\dots, a_{d-1}, -1+\sum_i a_i\right) \, ,
\]
Results of Bajo et al.~\cite[Theorem 2.10 and Corollary 5.4]{BajoBraunCodenottiHofscheierVindas} imply that the distribution arising as the coefficient vector of
\[
\frac{h^*(S_{M+1};z)-1}{h^*(S_{M+1};1)-1}
\]
is a limit of Ehrhart $h^*$-distributions.
For example, for positive integers $q$, consider the special case where $N=(d-2)q+1$ and $S_N$ is the $d$-simplex in one-row Hermite normal form with last row $(1,\dots,1,N)$.
From work of Bajo et al.~\cite[Theorems 2.10 and 3.1]{BajoBraunCodenottiHofscheierVindas}, it follows that $\operatorname{gcd}(M,N)=(d-2,(d-2)q+1)=1$ and 
\[
h^*(S;z) = 1 + q\sum_{i=2}^{d-1}z^i \, .
\]
Thus, 
\[
\hdis(S_N)=\left(\frac{1}{N}, 0 , \frac{q}{N}, \frac{q}{N},\ldots, \frac{q}{N},0\right) 
\]
and as $q\to \infty$ we obtain the limit point
\[
\left(0, 0 , \frac{1}{d-2}, \frac{1}{d-2},\ldots, \frac{1}{d-2},0\right) \, .
\]
\end{example}

\begin{example}
    For every choice of $m, d,k \in \Z $ such that $m\geq 1$, $d\geq 2,$ and $1\leq k\leq \lfloor \frac{d+2}{2} \rfloor$, Higashitani~\cite{higashitani} constructed a $d$-dimensional simplex with $h^\ast$-polynomial $1 + mz^k$ and normalized volume $m+1$.
    Thus, for fixed $d$ and fixed $k$ such that $1\leq k\leq \lfloor \frac{d+2}{2} \rfloor$, we can construct a sequence of simplices $\Delta_1, \Delta_2, \ldots$ such that $\Delta_m$ has $h^\ast$-polynomial $1 + mz^k$ and thus \[\hdis(\Delta_m) = \left(\frac{1}{m+1}, 0 ,0, \ldots, 0, \frac{m}{m+1}, 0 \ldots, 0\right) \]
    where $\frac{m}{m+1}$ appears in the $(k+1)$st entry. 
    Note that 
    \[\lim_{m\rightarrow \infty} \hdis(\Delta_m) = (0, \dots, 0, 1, 0, \dots, 0) = e_{k+1}. 
    \]
    So, in the  probability simplex in $\R^d$, each standard basis vector $e_{k+1}$ for $1\leq k \leq \lfloor \frac{d+2}{2} \rfloor$ is a limit point of $h^\ast$-distributions. 
\end{example}

\begin{example}
    \label{ex:dilates3d}
    Consider a random sample of $3$-dimensional simplices with volume between $8$ and $1000$, obtained by randomly sampling from matrices in Hermite normal form.
    For each simplex $P$, we compute the $h^*$-distribution of the first twenty-four dilates of $P$ and plot the resulting trajectory of the last two components of the $h^*$-distributions for the dilates, i.e., of $(h_2^*/V,h_3^*/V)$ where $P$ has normalized volume $V$.
    Figure~\ref{fig:trajectories} plots the resulting trajectories of the $h^*$-distributions of dilates, illustrating how the distributions for dilates all approach the same point, namely $(2/3,1/6)$.
\end{example}

\begin{figure}
    \centering
    \includegraphics[width=0.8\linewidth]{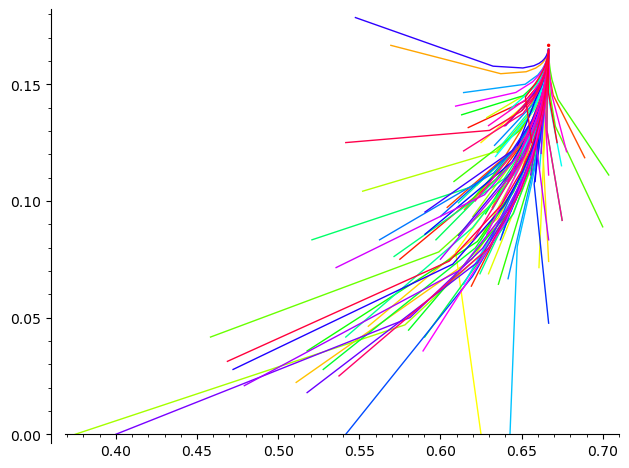}
    \caption{For $100$ random $3$-dimensional simplices, a plot of the trajectories of the last two coordinates of the $h^*$-distributions for the dilates of each simplex.
    The horizontal axis corresponds to the value $h^*_2/V$ while the vertical axis corresponds to the value $h^*_3/V$.
    This illustrates that the truncated $h^*$-distributions for dilates of $P$ approach the point $(4/6,1/6)$.}
    \label{fig:trajectories}
\end{figure}

As suggested by Example~\ref{ex:dilates3d}, we will prove next that the sequence of $h^*$-distributions arising from the sequence of polytopes $tP$ for $t=1,2,\ldots$ always converges to the same limit point. 
Let $\mathfrak{S}_d$ denote the symmetric group on $d$ elements, and let $\des(\pi)$ denote the descent statistic on $\pi\in \mathfrak{S}_d$.
The \emph{$d$-th Eulerian distribution} is the distribution associated to the polynomial
\[
\sum_{\pi\in \mathfrak{S}_d}z^{1+\des(\pi)} \, .
\]
We denote this distribution by $\eul_d$.

\begin{example}\label{ex:euleriandistribution}
For $d=3$, the Eulerian polynomial is $z+4z^2+z^3$, and thus 
\[
\eul_3=(0,1/6,4/6,1/6) \, .
\]
Hence, the cluster point $(4/6,1/6)$ for each trajectory shown in Example~\ref{ex:dilates3d} corresponds to a truncation of the Eulerian distribution.
\end{example}

Observe that in Corollary~\ref{cor:momentsdilates}, the term \((d+1)/2\) is the mean of the Eulerian distribution for the statistic $1+\des$, and \((d+1)/12\) is its variance for \(d\ge 2\).  
Thus, Corollary~\ref{cor:momentsdilates} predicts the following theorem.

\begin{theorem}\label{thm:dilates}
 Let $P$ be a $d$-dimensional lattice polytope.
 Then
 \[
 \lim_{t\to \infty}\hdis(tP)=\eul_d \, .
 \]
\end{theorem}

\begin{proof}
Beck and Stapledon~\cite{BeckStapledon} proved that if the roots of the Eulerian polynomial are 
\[
\rho_1<\rho_2<\cdots <\rho_{d+1}=0 \, ,
\]
and the roots of $h^*(tP;z)$ are
\[
\beta_{t,1}<\beta_{t,2}<\cdots <\beta_{t,d+1} \, ,
\]
then as $t\to \infty$ we have that $\beta_{t,i}\to \rho_i$.
Thus, we have that
\[
h^*(tP;z)=\prod_{i=1}^{d+1}\left(1-\frac{z}{\beta_{t,i}}\right)
\]
and hence,
\begin{align*}
    \hdis(tP;z) & = \left(\prod_{i=1}^{d+1}\left(1-\frac{z}{\beta_{t,i}}\right)\right) / \left(\prod_{i=1}^{d+1}\left(1-\frac{1}{\beta_{t,i}}\right)\right) \\
    & = \left(\prod_{i=1}^{d+1}\left(1-\frac{z}{\beta_{t,i}}\right)\right)\left(\prod_{i=1}^{d+1}\frac{\beta_{t,i}}{\beta_{t,i}-1}\right) \\
    & = \prod_{i=1}^{d+1}\frac{\beta_{t,i}-z}{\beta_{t,i}-1} \\
    & \rightarrow z\prod_{i=1}^{d}\frac{\rho_i-z}{\rho_i-1}
\end{align*}
as $t \to \infty$.
The final expression is the distribution for the Eulerian polynomial.
\end{proof}

\subsection{A combinatorial application}\label{ssec:combinatorialapplication}

We next discuss a combinatorial application of Theorem~\ref{thm:dilates}.
Let $P = [0,1]^d$ be the $d$-dimensional unit cube, and let $kP=[0,k]^d$ be the $k$-th dilate of $P$.
The Ehrhart series of $kP$ is 
\[
\sum_{t\geq 0}(tk+1)^dz^t\, ,
\]
and there is a beautiful connection between this series and permutation statistics involving wreath products of symmetric and cyclic groups.
We will give a condensed presentation of these permutation statistics in this subsection, see Beck and Braun~\cite{BeckBraunEulerian} and the references therein for more details and historical context.

Let $\omega$ be a primitive $k$-th root of unity, and consider the set of pairs $(\pi,\epsilon)$ where $\pi\in\mathfrak{S}_d$ and $\epsilon\in\{\omega^{k-1},\omega^{k-2}, \ldots,\omega^1,\omega^0\}^d$.
By convention, we define $\pi_{d+1}=d+1$, $\pi_0=0$, $\epsilon_{d+1}=1$, and $\epsilon_0=1$.
We will represent an element $(\omega^{c},j)\in \{\omega^{k-1},\omega^{k-2}, \ldots,\omega^1,\omega^0\}\times \{1,2,\ldots,d\}$ as $j^{c}$.
Totally order the elements of $\{\omega^{k-1},\omega^{k-2}, \ldots,\omega^1,\omega^0\}\times \{1,2,\ldots,d\}$ by $j^{c_j}<k^{c_k}$ if $c_j>c_k$ or if both $c_j=c_k$ and $j<k$.
For a pair $(\pi,\epsilon)=([\pi(1),\ldots,\pi(d)],(\omega^{c_1},\ldots,\omega^{c_d}))\in \mathfrak{S}_d\times \{\omega^{k-1},\ldots,\omega^0\}^d$, the \emph{descent set} is
\[
\mathrm{Des}(\pi,\epsilon):=\{j\in \{0,1,\ldots,d-1\}:\pi(j)^{c_j}>\pi(j+1)^{c_{j+1}}\}\, ,
\]
and the \emph{descent statistic} is
\[
\mathrm{des}(\pi,\epsilon):=|\mathrm{Des}(\pi,\epsilon)|\, .
\]
We call these values the \emph{Eulerian statistics for wreath products}, and they are connected to dilates of unit cubes via the identity
\[
\sum_{t\geq 0}(tk+1)^dz^t = \frac{\sum_{\substack{(\pi,\epsilon)\in \mathfrak{S}_d\times \{\omega^{k-1},\ldots,\omega^0\}^d}}z^{\mathrm{des}(\pi,\epsilon)}}{(1-z)^{d+1}}\, .
\]
Define the \emph{$(d,k)$-Eulerian polynomial for the wreath product} to be 
\[
A_{k,d}(z)=\sum_{t=0}^dA_{k,d,t}z^t:=\sum_{\substack{(\pi,\epsilon)\in \mathfrak{S}_d\times \{\omega^{k-1},\ldots,\omega^0\}^d}}z^{\mathrm{des}(\pi,\epsilon)} \, ,
\]
and note that the generating function identity above yields
\[
h^*([0,k]^d;z)=A_{k,d}(z) \, .
\]

\begin{example}\label{ex:dkeulerianmeanvariance}
Using Corollary~\ref{cor:momentsdilates} and the fact that $h^*([0,k]^d;z)=h^*(k\cdot [0,1]^d;z)$, we have that the mean and variance of the $(d,k)$-Eulerian distribution are $\frac{d+1}{2}-\frac{1}{k}$ and $(d+1)/12$, respectively.
Thus, these are distributions with identical variances that interpolate between the Eulerian distribution on $[0,d-1]$ (i.e., the generating function for the statistic $\des$) and the Eulerian distribution on $[1,d]$ (the distribution of the statistic $1+\des$).
\end{example}

The following theorem is an interesting consequence of the convergence of $h^*$-distributions for dilates.

\begin{theorem}
    \label{thm:eulerianwreath}
    For positive integers $d,t\geq 1$, we have
    \[
    \lim_{k\to \infty}\frac{A_{k,d,t}}{k^d} = A_{1,d,t-1}\, ,
    \]
    where $A_{1,d,t-1}$ is the number of permutations in $\mathfrak{S}_d$ with $t-1$ descents.
\end{theorem}

\begin{proof}
    Applying Theorem~\ref{thm:dilates} to the sequence of polytopes $([0,k]^d:k\geq 0)$ yields that the limit of their $h^*$-distributions is given by the Eulerian distribution $\eul_d$.
    However, the Eulerian distribution is simply the distribution for $A_{1,d}(z)$ shifted by one degree, and the result follows after canceling factors of $d!$.
\end{proof}

\section{Real-rooted distributions}\label{sec:realrooted}

\subsection{$h^*$-inequalities for real-rooted distributions}\label{ssec:tailboundapplications}

In this section we focus on real-rooted $h^*$-distributions.
We begin with the following theorem, which is an immediate consequence of the large deviation bounds~\eqref{eq:largedeviation} combined with $\nvol(P)=\sum_{j=0}^dh^*_j$.

\begin{theorem}
    \label{thm:realrootinequality}
    Let $P$ be an $h^*$-real-rooted lattice polytope with $\deg(h^*(P;z))=d$ and $X_P\sim \hdis(P)$.
    For $\E[X_P]+1\leq n\leq d$, we have
    \[
    \sum_{j=n}^d h^*_j\leq \left( \sum_{j=0}^d h^*_j \right) \left(\frac{\E[X_P]}{n}\right)^n\left(\frac{d-\E[X_P]}{d-n}\right)^{d-n} \, .
    \]
\end{theorem}

For classes of polytopes where $\E[X_P]$ can be determined, these become linear inequalities.
For example, for reflexive polytopes, we obtain the following; a similar result can be obtained for Gorenstein polytopes.

\begin{proposition}
    \label{prop:reflexiverealrootinequality}
    Let $P$ be an $h^*$-real-rooted reflexive polytope with $\deg(h^*(P;z))=d$ and $X_P\sim \hdis(P)$.
    For $\frac{d}{2}+1\leq n\leq d$, we have
    \[
    \sum_{j=n}^d h^*_j\leq \frac{1}{2^d}\left( \sum_{j=0}^d h^*_j \right) \left(\frac{d}{n}\right)^n\left(\frac{d}{d-n}\right)^{d-n} \, .
    \]
    An identical inequality holds for $0\leq m\leq \frac{d}{2}$ with $n$ replaced by $m$.
\end{proposition}

\begin{proof}
    Evaluate Theorem~\ref{thm:realrootinequality} using $\E[X_P]=d/2$ and simplify.
\end{proof}

One corollary of this is a lower bound on the normalized volume of a real-rooted reflexive polytope.
Further, as a consequence of Hoeffding's inqualities~\eqref{eq:hoeffding}, we have that the crosspolytopes have the most ``spread out'' $h^*$-distribution among $h^*$-real-rooted reflexive polytopes, where the family of \emph{crosspolytopes} $\Diamond_d:=\conv\{\pm e_i: 1\leq i\leq d\}$ have $h^*(\Diamond_d;z)=(1+z)^d$, and thus $\hdis(\Diamond_d)=P_{d,1/2}$.
We record these observations formally in the following corollary.

\begin{corollary}
    \label{cor:reflexiverealrootvolume}
    For $P$ an $h^*$-real-rooted reflexive polytope with $\deg(h^*(P;z))=d$, we have
    \[
    2^d\leq \sum_{j=0}^d h^*_j = \nvol(P)\, .
    \]
    The above inequality is sharp.
    For any $0\leq m\leq d/2$, we also have
    \[
    2^d\sum_{j=0}^mh^*_j \leq \left(\sum_{j=0}^d h^*_j\right) \left(\sum_{k=0}^m\binom{d}{k}\right) \, .
    \]
\end{corollary}

\begin{proof}
    For the first inequality, evaluate Proposition~\ref{prop:reflexiverealrootinequality} at $n=d$ and use the fact that $h^*_d=1$ for reflexive polytopes.
    The crosspolytopes $\Diamond_d$ demonstrate that the first inequality is sharp.
    The second inequality is a direct application of Hoeffding's inequalities~\eqref{eq:hoeffding}.
\end{proof}

\subsection{Asymptotic normality}\label{ssec:asymptoticnormality}

We next provide an asymptotic normality criterion for sequences of real-rooted $h^*$-distributions in terms of their Ehrhart polynomial coefficients.
Recall from~\eqref{eq:rrasymp} that a sequence of real-rooted distributions is asymptotically normal if and only if $\sqrt{\Var[X_d]}\to \infty$ as $d\to \infty$.

\begin{theorem}
    \label{thm:realrootedasymptoticnormality}
    Consider a sequence of $h^*$-real-rooted lattice polytopes $P_1,P_2,\ldots$ with dimensions $d_1<d_2<\cdots$ having $X_{P_j}\sim \hdis(P_j)$ and $L_{P_j}(t)=\sum_{i=0}^{d_j}c_i(j)t^i$.
    Suppose further that $B(P_j)\leq \nvol(P_j)$ for all $j$.
    The sequence $X_{P_j}$ is asymptotically normal if
    \begin{equation}\label{eq:asympsuff}
    \lim_{j\to \infty}\left( \frac{2}{d_j(d_j-1)}\frac{c_{{d_j}-2}(j)}{c_{d_j}(j)} + \frac{1}{12}(3d_j^2+d_j-5)\right)^{1/2} = \infty \, .
    \end{equation}
\end{theorem}

\begin{proof}
    Since $B(P_j)\leq \nvol(P_j)$, we have that $B(P_j)/2\nvol(P_j)\leq 1/2$.
    Thus, applying this inequality to the $\E[X_{P_j}]$ terms appearing in the variance formula from Theorem~\ref{thm:meanvariance}, we have that
    \[
    \Var[X_{P_j}]\geq \frac{2(d_j-2)!c_{d_j-2}(j)}{\nvol(P_j)}+\frac{(d_j+1)d_j}{2}-\frac{(d_j+1)(3d_j+2)}{12}-\frac{1}{4} \, . 
    \]
    Using $\nvol(P_j)=d_j!c_{d_j}(j)$ and simplifying algebraically, combined with~\eqref{eq:rrasymp}, yields the result.
\end{proof}

Observe that the inequality $B(P_j)\leq \nvol(P_j)$ is not automatic, e.g., it does not hold for unit cubes. 
Theorem~\ref{thm:realrootedasymptoticnormality} implies that for any sequence of $h^*$-real-rooted polytopes with $c_{d-2}/c_{d}$ having appropriate growth rate in $d$, we obtain asymptotic normality.
A sufficient condition for this is the following.

\begin{corollary}
    \label{cor:ehrposasympnormal}
    Suppose that $P_j$ is a sequence of $h^*$-real-rooted lattice polytopes of strictly increasing dimension such that each $P_j$ contains an interior lattice point and $c_{{d_j}-2}\geq 0$.
    Then the sequence $\hdis(P_j)$ is asymptotically normal.
\end{corollary}

\begin{proof}
    If $P$ contains an interior lattice point, call it $\mathbf{v}$, then taking pyramids over the facets of $P$ with apex $\mathbf{v}$ yields a triangulation of $P$.
    The normalized volume of each such pyramid is at least as large as the normalized volume of the corresponding facet, and thus $B(P)\leq \nvol(P)$.
    Given this and the real-rooted assumption, Theorem~\ref{thm:realrootedasymptoticnormality} applies.
    The non-negativity of $c_{d_{j}-2}$ yields that~\eqref{eq:asympsuff} grows at least linearly in $d_j$, hence diverges.
\end{proof}

\begin{example}
    \label{ex:zonotopesasympnormal}
    Consider a sequence $P_j$ of lattice zonotopes of strictly increasing dimension such that each $P_j$ contains a lattice point in its interior.
    Beck, Jochemko, and McCullough~\cite{beckjochemkomccullough} proved that zonotopes are $h^*$-real-rooted, and they satisfy $c_{d-2}\geq 0$ as a result of their general Ehrhart positivity, obtained as a corollary of Theorem~\ref{thm:stanleyzonotopes}.
    Any such sequence $P_j$ has an asymptotically normal sequence of $h^*$-distributions.
\end{example}

\begin{example}
    \label{ex:magicpositive}
    Consider again a sequence of lattice polytopes $P_j$ of strictly increasing dimension that are \emph{magic positive}, i.e., such that each Ehrhart polynomial is of the form $\sum_{i=0}^da_it^i(1+t)^{d-i}$ with $a_i\geq 0$.
    It is known~\cite{FHsurvey} that magic positivity implies both Ehrhart positivity $c_i\geq 0$ for all $i$ and that $h^*(P;z)$ is real-rooted.
    For any sequence where $P_j$ contains an interior lattice point for all $j$, the sequence of $h^*$-distributions is asymptotically normal.
    The zonotopes given in the previous example are known to be magic positive, as are Pitman-Stanley polytopes~\cite{luckandmagic}, Stasheff polytopes~\cite{Konoike}, some partial permutahedra~\cite{LiuZhang}, a special class of arbor polytopes~\cite{arbormagicpositive}, and certain generalized parking-function polytopes~\cite{HillLuoTrinhVindas}.
   
\end{example}

\section{Further directions}
\label{sec:furtherdirections}

Our study of $h^*$-distributions leads naturally to the following further directions for research.

\begin{question}\label{q:reflexiveasympnorm}
    Which reflexive polytopes are $h^*$-real-rooted with $c_{d-2}\geq 0$? 
    Any sequence of reflexive polytopes with strictly increasing dimensions that satisfy these conditions will have asymptotically normal $h^*$-distributions.
\end{question}

\begin{question}\label{q:boundaryvolume}
    Consider the (finite up to unimodular equivalence) set of $d$-dimensional lattice polytopes with fixed values of $B(P)$ and $\nvol(P)$.
    What is the behavior of $c_{d-2}$ among these polytopes?
    Equivalently, since all of these polytopes would have the same $\E(X_P)$ value, what is the behavior of the variances for the $h^*$-distributions for these polytopes?
    In a different direction, which lattice polytopes satisfy $B(P)\leq\nvol(P)$, and which of these are $h^*$-real-rooted?
\end{question}

\begin{question}
    \label{q:clusterpoints}
    For fixed dimension $d\geq 2$, which probability distributions arise as cluster points for $h^*$-distributions?
    What are sequences of polytopes with $h^*$-distributions that converge to a cluster point?  
\end{question}

\section{Tool and computational resource disclosure}

The authors did not use AI or LLM tools for any aspect of this research or the writing of this manuscript.
SageMath~10.9 was used for computations, available at \url{https://www.sagemath.org}.

\bibliographystyle{amsplain}
\bibliography{bibliography}

\providecommand{\bysame}{\leavevmode\hbox to3em{\hrulefill}\thinspace}
\providecommand{\MR}{\relax\ifhmode\unskip\space\fi MR }
\providecommand{\MRhref}[2]{%
  \href{http://www.ams.org/mathscinet-getitem?mr=#1}{#2}
}
\providecommand{\href}[2]{#2}
\begin{thebibliography}{10}

\bibitem{arbormagicpositive}
Christos~A. Athanasiadis, Qiqi Xiao, and Xue Yan, \emph{Lattice point enumeration of some arbor polytopes}, 2026, Preprint, \url{https://arxiv.org/abs/2603.11654}.

\bibitem{luckandmagic}
Nicolas Avila, Luis Ferroni, and Alejandro~H. Morales, \emph{Luck and magic for {P}itman-{S}tanley polytopes and parking functions}, 2026, https://arxiv.org/abs/2603.19194.

\bibitem{BajoBraunCodenottiHofscheierVindas}
Esme Bajo, Benjamin Braun, Giulia Codenotti, Johannes Hofscheier, and Andr\'es~R. Vindas-Mel\'endez, \emph{\href{https://link.springer.com/article/10.1007/s13366-025-00784-z}{Local \(h^*\)-polynomials for one-row {H}ermite normal form simplices}}, Beitr. Algebra Geom. (2025), published online.

\bibitem{Balletti-Higashitani}
Gabriele Balletti and Akihiro Higashitani, \emph{Universal inequalities in {Ehrhart} theory}, Isr. J. Math. \textbf{227} (2018), 843--859.

\bibitem{BeckBraunEulerian}
Matthias Beck and Benjamin Braun, \emph{Euler-{M}ahonian statistics via polyhedral geometry}, Adv. Math. \textbf{244} (2013), 925--954.

\bibitem{beckjochemkomccullough}
Matthias Beck, Katharina Jochemko, and Emily McCullough, \emph{{$h^\ast$}-polynomials of zonotopes}, Trans. Amer. Math. Soc. \textbf{371} (2019), no.~3, 2021--2042.

\bibitem{BeckRobins}
Matthias Beck and Sinai Robins, \emph{Computing the continuous discretely}, second ed., Undergraduate Texts in Mathematics, Springer, New York, 2015, Integer-point enumeration in polyhedra, With illustrations by David Austin.

\bibitem{BeckStapledon}
Matthias Beck and Alan Stapledon, \emph{On the log-concavity of {H}ilbert series of {V}eronese subrings and {E}hrhart series}, Math. Z. \textbf{264} (2010), no.~1, 195--207.

\bibitem{ampleframingflow}
Matias~von Bell, Benjamin Braun, Kaitlin Bruegge, Derek Hanely, Zachery Peterson, Khrystyna Serhiyenko, and Martha Yip, \emph{Triangulations of flow polytopes, ample framings, and gentle algebras}, Selecta Math. (N.S.) \textbf{30} (2024), no.~3, Paper No. 55, 34.

\bibitem{brandensurvey}
Petter Br\"and\'en, \emph{Unimodality, log-concavity, real-rootedness and beyond}, Handbook of enumerative combinatorics, Discrete Math. Appl. (Boca Raton), CRC Press, Boca Raton, FL, 2015, pp.~437--483.

\bibitem{braunsurvey}
Benjamin Braun, \emph{Unimodality problems in {E}hrhart theory}, Recent trends in combinatorics, IMA Vol. Math. Appl., vol. 159, Springer, [Cham], 2016, pp.~687--711.

\bibitem{endquasipoly}
Benjamin Braun, Christopher O'Neill, and Antwon Park, \emph{Eventually nondecreasing quasi-polynomials}, preprint, 2026.

\bibitem{BrunsGubeladze}
Winfried Bruns and Joseph Gubeladze, \emph{Polytopes, rings, and {$K$}-theory}, Springer Monographs in Mathematics, Springer, Dordrecht, 2009.

\bibitem{Ehrhart}
Eug\`ene Ehrhart, \emph{Sur les poly\`edres rationnels homoth\'etiques \`a{} {$n$}\ dimensions}, C. R. Acad. Sci. Paris \textbf{254} (1962), 616--618.

\bibitem{FHsurvey}
Luis Ferroni and Akihiro Higashitani, \emph{Examples and counterexamples in {E}hrhart theory}, EMS Surveys in Mathematical Sciences (2024), Published online first, https://doi.org/10.4171/emss/86.

\bibitem{Hibibook}
Takayuki Hibi, \emph{Algebraic combinatorics on convex polytopes}, Carslaw Publications, Glebe, 1992.

\bibitem{higashitani}
Akihiro Higashitani, \emph{Counterexamples of the conjecture on roots of {E}hrhart polynomials}, Discrete Comput. Geom. \textbf{47} (2012), no.~3, 618--623.

\bibitem{HillLuoTrinhVindas}
Charlie Hill, Ambrose Luo, Vu~Trinh, and Andr\'{e}s~R. Vindas-Mel\'{e}ndez, \emph{Lattice slices, {E}hrhart polynomials, and magic positivity of generalized parking-function polytopes}, 2026, in preparation.

\bibitem{Konoike}
Masato Konoike, \emph{A new class of magic positive {E}hrhart polynomials of reflexive polytopes}, S\'em. Lothar. Combin. \textbf{93B} (2025), Art. 148, 11. \MR{4975683}

\bibitem{LagariasZiegler}
Jeffrey~C. Lagarias and G\"unter~M. Ziegler, \emph{Bounds for lattice polytopes containing a fixed number of interior points in a sublattice}, Canad. J. Math. \textbf{43} (1991), no.~5, 1022--1035.

\bibitem{LiuZhang}
Feihu Liu and Zihao Zhang, \emph{Magic positivity for the {E}hrhart polynomials of partial permutohedra}, 2026, Preprint, \url{https://arxiv.org/abs/2607.03854}.

\bibitem{pitman}
Jim Pitman, \emph{Probabilistic bounds on the coefficients of polynomials with only real zeros}, J. Combin. Theory Ser. A \textbf{77} (1997), no.~2, 279--303.

\bibitem{Schrivjer}
Alexander Schrijver, \emph{Theory of linear and integer programming}, Wiley-Interscience Series in Discrete Mathematics, John Wiley \& Sons, Ltd., Chichester, 1986, A Wiley-Interscience Publication.

\bibitem{Scott}
P.R. Scott, \emph{On convex lattice polygons}, Bull. Austral. Math. Soc. \textbf{15} (1976), no.~3, 395--399.

\bibitem{StanleyDecompositions}
Richard~P. Stanley, \emph{Decompositions of rational convex polytopes}, Ann. Discrete Math. \textbf{6} (1980), 333--342.

\bibitem{stanleyzonotopes}
\bysame, \emph{A zonotope associated with graphical degree sequences}, Applied geometry and discrete mathematics, DIMACS Ser. Discrete Math. Theoret. Comput. Sci., vol.~4, Amer. Math. Soc., Providence, RI, 1991, pp.~555--570.

\end{thebibliography}
\end{document}